\documentclass[12pt,oneside]{amsart}

\usepackage[utf8]{inputenc}
\usepackage{amsmath,amssymb, amsthm, graphicx, tikz, mathtools, mathrsfs, comment, setspace, microtype,dsfont,float,hyperref,tikz-cd,array,fancyhdr,enumitem,accents,xspace,braids}
\usepackage[margin=1in]{geometry}

\makeatletter
\g@addto@macro\th@plain{\thm@headpunct{:}}
\makeatother      
\theoremstyle{plain}

\makeatletter
\g@addto@macro\@floatboxreset\centering
\makeatother

\makeatletter
\def\moverlay{\mathpalette\mov@rlay}
\def\mov@rlay#1#2{\leavevmode\vtop{%
		\baselineskip\z@skip \lineskiplimit-\maxdimen
		\ialign{\hfil$\m@th#1##$\hfil\cr#2\crcr}}}
\newcommand{\charfusion}[3][\mathord]{
	#1{\ifx#1\mathop\vphantom{#2}\fi
		\mathpalette\mov@rlay{#2\cr#3}
	}
	\ifx#1\mathop\expandafter\displaylimits\fi}
\makeatother

\newtheorem{theorem}{Theorem}[section]
\newtheorem{corollary}[theorem]{Corollary}
\newtheorem{lemma}[theorem]{Lemma}
\newtheorem{proposition}[theorem]{Proposition}

\newtheorem*{theorem*}{Theorem}
\newtheorem*{lemma*}{Lemma}
\newtheorem{definition}[theorem]{Definition}
\let\olddefinition\definition
\renewcommand{\definition}{\olddefinition\normalfont}

\newtheorem*{definition*}{Definition}
\newtheorem*{remark*}{Remark}

\newcommand{\R}{\mathbb{R}}
\newcommand{\C}{\mathbb{C}}
\newcommand{\Z}{\mathbb{Z}}

\newcommand{\N}{\mathbb{N}}

\newcommand{\Abs}[1]{\left\lvert #1 \right\rvert}
\newcommand{\tens}{\otimes}

\newcommand{\isomorphic}{\cong}

\newcommand{\uniformlyto}{\rightrightarrows}
\newcommand{\dsum}{\oplus}

\renewcommand{\phi}{\varphi}

\newcommand{\Var}{\text{\textnormal{Var}}}

\makeatletter
\let\@@pmod\pmod
\DeclareRobustCommand{\pmod}{\@ifstar\@pmods\@@pmod}
\def\@pmods#1{\mkern4mu({\operator@font mod}\mkern 6mu#1)}
\makeatother

\newcommand\restr[2]{{% we make the whole thing an ordinary symbol
		\left.\kern-\nulldelimiterspace % automatically resize the bar with \right
		#1 % the function
		\vphantom{\big|} % pretend it's a little taller at normal size
		\right|_{#2} % this is the delimiter
	}}
\begin{document}
\author{Max Zhou}
\address{\noindent Department of Mathematics, Indiana University \\
	Bloomington, IN 47405, USA}
\email{maxzhou@indiana.edu}

\title[Return Probabilities of Random Walks]{Return Probabilities of Random Walks}
\date{December 11, 2015}

\begin{abstract}
	Associated to a random walk on $\Z$ and a positive integer $n$, there is a return probability of the random walk returning to the origin after $n$ steps. An interesting question is when the set of return probabilities uniquely determines the random walk. There are trivial situations where two different random walks have the same return probabilities. However, among random walks in which these situations do not occur, our main result states that the return probabilities do determine the random walk. As a corollary, we will obtain a result dealing with the representation theory of $\mathrm{U}(1)$.  
\end{abstract}



\maketitle

\section{Introduction}

The object of this paper is to first determine when a random walk is determined by its return probabilities, and then deduce a result in the representation theory of $\mathrm{U}(1)$. 

First, we will introduce relevant definitions related to random walks. Let $\{a_n\}$, $n \in \Z$, be a 1-dimensional random walk on $\Z$, where $a_n$ is the probability of moving $n$ steps in one unit of time. In this paper, we will only consider random walks for which
\begin{align*}
\limsup_{n \to \infty} \; a_n^{\enspace 1/n} < 1\;.
\end{align*}
All of our results will implicitly assume this condition. This condition is related to convergence of a Laurent series with coefficients given by the $a_n$.
 A random walk is \textit{\textbf{symmetric}} if $a_n = a_{-n}$ for all $n$. The set of \textit{\textbf{step sizes}} for a random walk $\{a_n\}$ is the set $\{\Abs{n} \in \N: a_n \neq 0 \}$. A random walk is \textit{\textbf{primitive}} if the set of step sizes generates $\Z$, and \textit{\textbf{compactly supported}} if the set of step sizes is finite.  The $k^{th}$ \textit{\textbf{return probability}} of a random walk is the probability that the position at time $k$ is $0$, i.e. $\mathrm{Pr}[x_1 + x_2 + \cdots + x_k = 0]$, where the $x_i$ are i.i.d discrete random variables with $\mathrm{Pr}[x_i = n] = a_n$.    

In general, a random walk is not uniquely determined by its return probabilities. 
For example, a random walk has all return probabilities 0 if $a_n = 0$ for $n \leq 0$ or $n \geq 0$. Also, if $\{a_n\}$ is a random walk and $c \in \Z \setminus \{0\}$, then the random walk $\{b_n\}$ given by $b_n = a_{c n}$ has the same return probabilities, but in general is not the same random walk. These situations can be avoided if we require that  $\{a_n\}$ and $\{b_n\}$ be symmetric and primitive.

Our main theorem is that under these hypotheses, a random walk is uniquely determined by its return probabilities: 

\begin{theorem}
	\label{thm:main}
	If $\{a_n\}$ and $\{b_n\}$ are symmetric, primitive random walks on $\Z$ whose return probabilities are all equal, then $\{a_n\} = \{b_n\}$.  
\end{theorem}

As a corollary, we obtain a result in the representation theory of $\mathrm{U}(1)$: 

\begin{corollary}
	\label{cor:rep_theory}
	Let $\rho: \mathrm{U}(1) \to \mathrm{GL}(V)$ be a complex, finite-dimensional, self-dual, faithful representation of $\mathrm{U}(1)$. Then, $\rho$ is uniquely determined by the function 
	\begin{align*}
	n \mapsto \dim (V^{\tens n})^{\mathrm{U}(1)}
	\end{align*}
\end{corollary}

\section{Proof of Theorem \ref{thm:main}}

First, it is clear that for any compact supported random walk $\{a_n\}$, the $n^{th}$ return probability is the constant coefficient of the Laurent polynomial $\left( \sum_{k = - \infty}^{ \infty} a_k t^k \right)^n$. We may also obtain this coefficient as the constant term of a Fourier series by substituting $t = e^{i x}$ and integrating from $0$ to $2 \pi$. To deal with the case where infinitely many of the $a_n$ are non-zero, we first   determine when the generating functions are analytic:

\begin{proposition}
	\label{prop:analyticity}
	Let $\{a_n\}$ be a random walk, and let $g(t) = \sum_{k = - \infty}^{\infty} a_k t^k$. Then, $g(t)$ is complex analytic in some annulus, $A = \{z \in \C: a < \Abs{z} < b \}$, where $a < 1 < b$. Further, $f(x) := g(e^{ix}) = \sum_{- \infty}^{\infty} a_k e^{i k x}$ is complex analytic in some neighborhood of $0$ in $\C$. Finally, $f_m(x) := \sum_{k = - m}^{m} a_k e^{ikx}$ converges normally to $f(x)$ on this neighborhood. 
\end{proposition}
\begin{proof}
	For the first point, note that $g(t) = \sum_{k = - \infty}^{\infty} a_k t^k = g_0(t) + g_1(t)$, where $g_0(t) = \sum_{k = - \infty}^{-1} a_k t^k$ and $g_1(t) = \sum_{k = 0}^{\infty} a_k t^k$. Let $\limsup_{n \to \infty} \; a_n^{\enspace 1/n} < c < 1$, for some $c > 0$. It follows from the root test that $g_0(t)$ converges absolutely on the neighborhood of infinity $\{z \in \C : \Abs{z} > c\}$ and $g_1(t)$ converges absolutely on the open disk $\{z \in \C : \Abs{z} < 1/c \}$. It follows from an application of the Weierstrass $M$-test that the partial sums of both $g_0(t)$ and $g_1(t)$ converge normally to their respective functions in their respective regions of convergence. Hence, $g_0(t)$ and $g_1(t)$ are analytic in their regions of convergence and $\sum_{k = - m}^{m} a_k t^k$ converges normally to $g(t)$ in the annulus $A = \{z \in \C : c < \Abs{z} < 1 / c \}$. Thus, $g(t) = g_0(t) + g_1(t)$ is analytic in $A $. As $c < 1$, then $c < 1 < 1 / c$. 
	
	For the second point, it is clear there is some neighborhood, $U$, of $0$, where $e^{ix}$ maps $U$ into $A$. Thus, the chain rule implies that $f(x) = g(e^{ix}) =  \sum_{- \infty}^{\infty} a_k e^{i k x}$ is complex analytic in $U$. Normal convergence of $\sum_{k = - m}^{m} a_k t^k$ to $g(t)$ implies normal convergence of $f_m(x) := \sum_{k = - m}^{m} a_k e^{ikx}$ to $f(x) = g(e^{ix})$ for $x \in U$.
\end{proof}

Now, we can show the integral that gives return probabilities for compactly supported random walks also gives return probabilities for all random walks:

\begin{lemma}
	\label{lem:fourier}
	Let $\{a_n\}$ be a random walk, and let $f(x) = \sum_{k = - \infty}^{ \infty} a_k e^{i k x}$. If $\{a_n\}$ is a symmetric random walk, then for $x \in \R$, $f(x) \in \R$ and $\Abs{f(x)} \leq 1$. The $n^{th}$ return probability of $\{a_n\}$, $c_n$, is given by: 
	\begin{align*}
	c_n = \frac{1}{2 \pi} \int_{- \pi}^{\pi} f(x)^n \; dx
	\end{align*}
\end{lemma}
\begin{proof}
	Let $f_m(x) = \sum_{k  = -m}^{m} a_k e^{i k x}$. It follows from the convergence of $\sum_{k = - \infty}^{\infty} a_k$ and the Weierstrass $M$-test that $f_m(x) \uniformlyto f(x)$ on $[- \pi, \pi]$.
	
	For the first point, if $\{a_n\}$ is symmetric, then $f_m(x) = a_0 + 2 \sum_{k = 1}^{m} a_k \cos(k x)$.  As $f_m(x)$ is real valued for every $m$, then $f(x)$ is also real valued.
	
	Next, it follows from an easy application of the triangle inequality that for every $m \in \N$ and $x \in \R$, $\Abs{f_m(x)} \leq \sum_{k = - m}^{m} a_k \leq 1$. Letting $m \to \infty$, then $\Abs{f(x)} \leq 1$.
	
	For the last point, let
	\begin{align*}
	c_{m, n} = \frac{1}{2 \pi} \int_{- \pi}^{\pi} f_m(x)^n \; dx
	\end{align*}
	
	Since $c_{m, n}$ is also the constant coefficient of $( \sum_{k = - m}^{m} a_k t^k )^n$, then $c_{m, n} = \mathrm{Pr}[x_1 + \cdots + x_n = 0 \text{ and } \Abs{x_i} \leq m \text{ for all } i]$, where the $x_i$ are i.i.d. discrete random variables with $\mathrm{Pr}[x_i = n] = a_n$.
	
	Now, we will show that, for every $n$, $\lim\limits_{m \to \infty} c_{m, n} = c_n$. First, observe the following inequalities: 
	\begin{align*}
	c_n - c_{m, n} & = \mathrm{Pr}[x_1 + \cdots + x_n = 0 \text{ and } \Abs{x_i} > m \text{ for some } i]
	\\ & \leq \sum_{i = 1}^{n} \mathrm{Pr}[\Abs{x_i} > m]
	\\ &= n \cdot \mathrm{Pr}[\Abs{x_1} > m]
	\end{align*}
	When $n$ is fixed and $m \to \infty$, then the last expression goes to 0. Hence, $\lim\limits_{m \to \infty} c_{m, n} = c_n$. 
	
	On the other hand, since $f_m(x) \uniformlyto f(x)$ on $[- \pi, \pi]$ (and hence the $f_m(x)$ are uniformly bounded), then for every $n$, $f_m(x)^n \uniformlyto f(x)^n$ on $[- \pi, \pi]$. Thus, we may interchange limit and integral, and the following equalities prove the second point: 
	\begin{align*}
	c_n = \lim\limits_{m \to \infty} c_{m, n} 
	= \lim\limits_{m \to \infty} \frac{1}{2 \pi} \int_{- \pi}^{\pi} f_m(x)^n \; dx 
	= \frac{1}{2 \pi} \int_{- \pi}^{\pi} \lim\limits_{m \to \infty}  f_m(x)^n \; dx 
	= \frac{1}{2 \pi} \int_{- \pi}^{\pi} f(x)^n \; dx 
	\end{align*}
\end{proof}

As $n$ becomes large, only the neighborhoods of points where $\Abs{f(x)} = 1$ contribute significantly to the integral in Lemma \ref{lem:fourier}. When $\{a_n\}$ is a primitive random walk, only the obvious points satisfy this condition: 

\begin{lemma}
	\label{lem:extrema}
	Let $\{a_n\}$ be a symmetric, primitive random walk. Let $f(x) = \sum_{k = - \infty}^{ \infty} a_k e^{i k x} = a_0 + 2 \sum_{k = 1}^{\infty} a_k \cos(k x)$. If there is an even step size, then $\Abs{f(x)} = 1$ only when $x = 2 \pi k$ for some $k \in \Z$. On the other hand, if the step sizes of $\{a_n\}$ are all odd, then $\Abs{f(x)} = 1$ only when $x = m \pi$ for some $m \in \Z$.
\end{lemma}
\begin{proof}
	When there is an even step size, then $f(x)$ can never be $-1$. For if it was, then for each step size $k_i$, $\cos(k_i x) = -1$. Suppose that $k_1$ is odd, while $k_2$ is even. Then, $\cos(k_1 x) = -1$ when $x = (\pi + 2 \pi a_1)/ k_1$ for some $a_1 \in \Z$. If $k_2 = 0$, then $\cos(k_2 x) = -1$ cannot happen, and we are done. If $k_2 \neq 0$, then similarly from $\cos(k_2 x) = -1$, we have $x = (\pi + 2 \pi a_2) / k_2$ for some $a_2 \in \Z$. Equating both expressions for $x$ and simplifying, we have $k_1 \pi + 2 \pi a_2 k_1 = k_2 \pi + 2 \pi a_1 k_2$. However, the left hand side is an odd multiple of $\pi$, while the right hand side is an even multiple of $\pi$, a contradiction. 
	
	Next, whenever the step sizes generate $\Z$, $f(x) = 1$ only when $x$ is a multiple of $2 \pi$. Again, noting that for each non-zero step size $k_i$, $\cos(k_i x) = 1$, we obtain the equations $x = 2 \pi a_i / k_i$ for each $i$, so then all of the expressions of the form $a_i / k_i$ are equal. If there is one $k_j$ that does not divide $a_j$, then none of the $k_i$'s divide $a_i$. In this case, we may reduce these fractions to assume that $(a_i, k_i) = 1$, while still assuming that the $k_i$ generate $\Z$. Then, for any $i, j$, $a_i k_j = k_i a_j$. The conditions $(a_i, k_i) = 1$ and $(a_j, k_j) = 1$ imply that $a_i \mid a_j$ and $a_j \mid a_i$, respectively. Hence, all of the $a_i$'s are the same, which implies all of the $1 / k_i$ are equal. This clearly cannot happen if the $k_i$ are relatively prime, unless $\{a_n\}$ has only one step size. Thus, when $\{a_n\}$ has more than one step size, $k_i \mid a_i$ for every $i$, and $x$ is a multiple of $2 \pi$.  If $\{a_n\}$ has only one step size, it follows from symmetry and primitivity that $a_{1} = a_{-1} = 1/2$ and $a_n = 0$ for all other $n$. Then, $f(x) = e^{- i x} / 2 + e^{i x} / 2 = \cos(x)$, which clearly is only $1$ when $x = 2 \pi$. 
	
	Finally, if all of the step sizes are odd, then $f(x) = -1$ only when $x$ is an odd multiple of $\pi$. For each non-zero step size $k_i$, $\cos(k_i x) = -1$, so then $k_i x = \pi + 2 \pi a_i$ for some $a_i \in \Z$. It follows that $x = \pi (1 + 2 a_i) / k_i$, and hence expressions of the form $(1 + 2 a_i) / k_i$ are equal. Using the same logic as in the previous case, we see that it must be that $k_i \mid (1 + 2 a_i)$ for all $i$. As $x = \pi(1 + 2  a_i) / k_i$, then $x$ must be an odd multiple of $\pi$.  
\end{proof}

The dichotomy of $\{a_n\}$ either having either only odd step sizes or having an even step size will persist throughout the several of the following results. For brevity, we make the following definitions:

\begin{definition}
	A symmetric, primitive random walk is \textit{\textbf{Type 1}} if it has an even step size, and \textit{\textbf{Type 2}} if it has an odd step size. 
\end{definition}

In particular, the next result shows it is possible to determine if $\{a_n\}$ is Type 1 or Type 2 from its return probabilities.

\begin{lemma}
	\label{lem:odd_steps}
	A random walk $\{a_n\}$ is Type 2 if and only if every odd return probability is 0.
\end{lemma}
\begin{proof}
	$(\Rightarrow)$: If $a_n$ has only odd step sizes, then any sum of an odd number of these step sizes is odd, and hence cannot be 0. 
	
	$(\Leftarrow)$: It is equivalent to show the contrapositive: if $\{a_n\}$ has an even step size, then there is a non-zero odd return probability. Either $a_0$ is the only non-zero coefficient of an even step size, or it is possible to pick $a_{k}$ and $a_{j}$, where $k$ and $j$ have different signs and different parity. In the first case, any return probability is non-zero. In the second case, $k + j$ is odd and the $(k + j)^{th}$ return probability is non-zero, as it is possible to take $j$ steps $\Abs{k}$ times, followed by $k$ steps $\Abs{j}$ times. 
\end{proof}

We will change variables in the integral in Lemma \ref{lem:fourier} to relate the return probabilities of a random walk with the moments of $(f^{-1})'(y)$ on a half-neighborhood of $1$. 

First, we will describe $f^{-1}(y)$ and its derivative: 

\begin{proposition}
	\label{prop:h_function}
	If $\{a_n\}$ is a symmetric random walk, then $f(x) = \sum_{k = - \infty}^{ \infty} a_k e^{i k x}$ can be written as $f(x) = h(x^2)$, where $h(x)$ is complex analytic in a neighborhood of 0 in $\C$, $h(0) = 1$, and $h'(0) < 0$. It follows that $f^{-1}(y)$ is defined on $[1 - \eta, 1]$ and differentiable in $(1 - \eta, 1)$, for some $\eta > 0$. Further, $(f^{-1})'(y) = (1 - y)^{- 1/ 2} q(y)$, where $q(y)$ is complex analytic in a neighborhood of $1$.
\end{proposition}
\begin{proof}
	Let $f_m(x) = \sum_{k = -m}^{m} a_k e^{i k x}$. From Proposition \ref{prop:analyticity}, $f_m(x)$ converges normally to $f(x)$ on some neighborhood of $0$ in $\C$. 
	
	Then, for every $n^{th}$ derivative, $\lim\limits_{m \to \infty} f_m^{(n)}(0) = f^{(n)}(0)$. As $f_m^{(n)}(0) = 0$ for $n$ odd, then $f^{(n)}(0) = 0$ for $n$ odd. Since $f(x)$ is analytic in a neighborhood of 0, then $f(x)$ has a power series expansion in a neighborhood of 0, $f(x) = \sum_{n = 0}^{\infty} \alpha_n x^n$. As $\alpha_{2k} = 0$, then $f(x) = h(x^2)$, where $h(x) = \sum_{n = 0}^{\infty} \alpha_{2n} x^n$ is also analytic in a neighborhood of $0$. For the specific values $h(0) = 1$ and $h'(0) < 0$, consider the following equalities: 
	\begin{align*}
	h(0) &= f(0) = \lim\limits_{m \to \infty} f_m(0) = \lim\limits_{m \to \infty} \sum_{k = - m}^{m} a_k = 1
	\\ h'(0) &= \alpha_2 = \frac{f''(0)}{2} = \lim\limits_{m \to \infty} \frac{f_m''(0)}{2} = \lim\limits_{m \to \infty} -\frac{1}{2} \sum_{k = -m}^{m} k^2 a_k < 0
	\end{align*}
	
	It follows from the Lagrange inversion theorem, stated as Burmann's theorem in \cite[p.~129]{whit2003}, that $h^{-1}(y)$ is defined and complex analytic near $y = 1$, with $(h^{-1})'(1) < 0$. Hence, we have $f^{-1}(y) = \sqrt{h^{-1}(y)}$, which is real and defined when $y = h(x)$, for $x \in [0, \epsilon]$ for some $\epsilon > 0$. Thus, $f^{-1}(y)$ is defined on $[1 - \eta, 1]$ for $1 - \eta = h(\epsilon)$. As $f^{-1}(1) = 0$, then $(f^{-1})'(y)$ exists on $(1 - \eta, 1)$.
	
	Expanding $h^{-1}(y)$ in a power series near 1, we have $h^{-1}(y) = \sum_{k = 0}^{\infty} b_k (y - 1)^k$, where $b_0 = 0$ and $b_1 = 1 / \alpha_2 < 0$. Hence, $h^{-1}(y) = (1 - y) p(y)$, where $p(y)$ is complex analytic near $y = 1$ and $p(1) = - b_1 > 0$. By making $\eta$ smaller, we may assume that $p(y)$ is complex analytic in $[1 - \eta, 1]$. It follows from $p(1) > 0$ and invertibility of $h^{-1}(y) = (1 - y) p(y)$ on $[1 - \eta, 1]$ that $p(y) > 0$ on $[1 - \eta, 1]$.
	
	For $y \in [1 - \eta, 1]$, $f^{-1}(y) = \sqrt{h^{-1}(y)}= \sqrt{1 -y} \sqrt{p(y)}$, where the square roots are principal square roots. As $p(y) > 0$ on $[1 - \eta, 1]$, then $\sqrt{p(y)}$ is complex analytic on $[1 - \eta, 1]$. Differentiating $f^{-1}(y)$ and rearranging terms produces $q(y)$. 
\end{proof}

%\begin{proposition}
%	If $f(x) = h(x^2)$, where $h(x)$ extends to be complex analytic, $h(0) = 1$, and $h'(0) < 0$, then $(f^{-1})'(y) = (1 - y)^{- 1/ 2} q(y)$ in $[1 - \eta, 1]$ for some $q(y)$ that extends to be complex analytic in a neighborhood of $1$.
%\end{proposition}
%\begin{proof}
%	
%\end{proof}
%
%If $\{a_n\}$ is a symmetric random walk, then  can be written as $f(x) = a_0 + \sum_{k = 1}^{\infty} 2 a_k \cos(k x)$. By examining the power series expansion of $f(x)$ around $0$, we see that $f(x) = h(x^2)$, where $h(x)$ extends to be complex analytic. Further, $h(0) = 1$, and from primitivity of $\{a_n\}$, $f(x) \neq 1$, so $h'(0) < 0$. Hence, $f$ on $[0, \epsilon]$ has a local inverse that is continuous in $[1 - \eta, 1]$ and differentiable in $(1 - \eta, 1)$ for some $\epsilon, \eta > 0$. We now relate the moments of $(f^{-1})'(y)$ on $[1 - \eta, 1]$ to the return probabilities of $\{a_n\}$:

\begin{proposition}
	\label{prop:estimate_1}
	
	Let $\{a_n\}$ be a symmetric, primitive random walk. Let $f(x) = \sum_{k = - \infty}^{ \infty} a_k e^{i k x}$ have an inverse that is continuous in $[1 - \eta, 1]$ and continuously differentiable in $(1 - \eta, 1)$ for some $\eta > 0$. Let $c_n$ be the $n^{th}$ return probability. 
	
	If $\{a_n\}$ is Type 1, then
	\begin{align*}
	c_n + \frac{1}{\pi} \int_{1 - \eta}^{1} y^n \; (f^{-1})' (y) \; dy = o(e^{- \mu n})
	\end{align*}
	for some $\mu > 0$.
	
	If $\{a_n\}$ is Type 2, then
	\begin{align*}
	c_{2n} + \frac{2}{\pi} \int_{1 - \eta}^{1} y^{2n} \; (f^{-1})' (y) \; dy = o(e^{- 2\mu n})
	\end{align*}
	for some $\mu > 0$.
\end{proposition}
\begin{proof}
	From Lemma \ref{lem:fourier}, 
	\begin{align*}
	c_n = \frac{1}{2 \pi} \int_{- \pi}^{\pi} f(x) ^n \; dx
	\end{align*}
	
	Let $\epsilon = f^{-1}(1 - \eta)$.
	
	When $\{a_n\}$ is Type 1, then from Lemma \ref{lem:extrema}, there is some $\alpha < 1$ where $\Abs{f(x)} \leq \alpha$ for $x \in [- \pi, - \epsilon] \cup [\epsilon, \pi]$. 
	
	From the ML-inequality and $f(x)$ being even, the following inequalities hold: 
	\begin{align*}
	\Abs{c_n - \frac{1}{\pi} \int_{0}^{\epsilon} f(x)^n \; dx } 
	= \frac{1}{2 \pi} \Abs{\int_{- \pi}^{- \epsilon} f(x)^n \; dx \; + \; \int_{\epsilon}^{\pi} f(x)^n \; dx } 
	< \alpha^n = e^{\log(\alpha) n}
	\end{align*}
	
	Setting $\mu = - \log(\alpha) / 2$, and changing variables $x = f^{-1}(y)$ to obtain: 
	\begin{align*}
	- \int_{0}^{\epsilon} f(x)^n \; dx =   \int_{1 - \eta}^{1} y^n (f^{-1})' (y) \; dy \;,
	\end{align*}
	which proves the case when $\{a_n\}$ is Type 1

	When $\{a_n\}$ is Type 2, then $f(x + \pi) = - f(x)$. Hence, 
	\begin{align*}
	c_n = 
	\begin{dcases*}
	0 & if $n$ is odd \\
	\frac{1}{\pi} \int_{- \pi / 2}^{\pi / 2} f(x)^n \; dx & if $n$ is even
	\end{dcases*}
	\end{align*}
	
	In order to make similar estimates as in the previous case, we restrict attention to only the even return probabilities. It follows from Lemma \ref{lem:extrema} there is some $\alpha < 1$ where $\Abs{f(x)} \leq \alpha$ for $x \in [- \pi / 2, - \epsilon] \cup [\epsilon, \pi / 2]$. From the ML-inequality and $f(x)$ being even, the following inequalities hold:
	\begin{align*}
	\Abs{c_{2n} - \frac{2}{\pi} \int_{0}^{\epsilon} f(x)^{2n} \; dx } 
	= \frac{1}{\pi} \Abs{\int_{- \pi / 2}^{- \epsilon} f(x)^{2n} \; dx \; + \; \int_{\epsilon}^{\pi / 2} f(x)^{2n} \; dx } 
	< \alpha^{2n} = e^{2 \log(\alpha) n}
	\end{align*}
	
	Setting $\mu = - \log(\alpha) / 2$, and changing variables $x = f^{-1}(y)$ to obtain: 
	\begin{align*}
	- \int_{0}^{\epsilon} f(x)^n \; dx =   \int_{1 - \eta}^{1} y^n (f^{-1})' (y) \; dy \;,
	\end{align*}
	which proves the case when $\{a_n\}$ is Type 2.
\end{proof}

The next proposition follows easily from the previous proposition:

\begin{proposition}
	\label{prop:estimate_2}
	Let $\{a_n\}$ and $\{b_n\}$ be symmetric, primitive random walks that have the same return probabilities. Let $f(x) = \sum_{k = - \infty}^{ \infty} a_k e^{i k x}$ and $g(x) = \sum_{k = - \infty}^{\infty} b_k e^{i k x}$ have inverses that are continuous in $[1 - \eta, 1]$ and continuously differentiable in $(1 - \eta, 1)$ for some $\eta > 0$. 
	
	If $\{a_n\}$ and $\{b_n\}$ are both Type 1, then 
	\begin{align*}
		 \int_{1 - \eta}^{1} y^n  \left( (f^{-1}) ' (y) - (g^{-1}) '(y) \right)  \; dy = o(e^{- \mu n}) 
	\end{align*}
	for some $\mu > 0$. 
	
	If $\{a_n\}$ and $\{b_n\}$ are both Type 2, then 
	\begin{align*}
	\int_{1 - \eta}^{1} y^{2n}  \left( (f^{-1}) ' (y) - (g^{-1}) '(y) \right)  \; dy = o(e^{- 2 \mu n}) 
	\end{align*}
	for some $\mu > 0$. 
\end{proposition}
\begin{proof}
	We will just prove the first case, as the second case follows from nearly identical estimates.
	Let $c_n$ be the return probabilities of $\{a_n\}$ and $\{b_n\}$. From Proposition \ref{prop:estimate_1} and triangle inequality, the following inequalities hold:
	\begin{align*}
	&\frac{1}{\pi} \Abs{\int_{1 - \eta}^{1} y^n \left( (f^{-1}) ' (y) - (g^{-1}) '(y) \right)  \; dy} 
	\\ &= \Abs{c_n + \frac{1}{\pi} \int_{1 - \eta}^{1} y^n (f^{-1})' (y) \; dy - \left(  c_n + \frac{1}{\pi} \int_{1 - \eta}^{1} y^n (f^{-1})' (y) \; dy \right) } 
	\\ &\leq o(e^{- \mu n}) + o(e^{- \mu' n}) = o(e^{- \min(\mu, \mu') n})
	\end{align*}
\end{proof}

Proposition \ref{prop:estimate_3} will show that the asymptotic behavior of the integrals in Proposition \ref{prop:estimate_2} can only hold when $(f^{-1})'(y) = (g^{-1})'(y)$ on $(1 - \eta, 1)$. 

First, assuming this has been shown, we will prove Theorem \ref{thm:main}: 
\begin{proof}[Proof of Theorem \ref{thm:main}]
	Let $f(x) = \sum_{k = - \infty}^{\infty} a_k e^{i k x}$ and let $g(x) = \sum_{k = - \infty}^{\infty} b_k e^{i k x}$. If the return probabilities of $\{a_n\}$ and $\{b_n\}$ are all equal, then it follows from Lemma \ref{lem:odd_steps} that either they both have only odd step sizes or they both have an even step size. In either of these cases, it follows from Proposition \ref{prop:estimate_2} that the hypotheses of Proposition \ref{prop:estimate_3} hold, and hence $(f^{-1})'(y) = (g^{-1})'(y)$ on $(1 - \eta, 1)$. As $f^{-1}(1) = g^{-1}(1) = 0$, then $f^{-1}(y) = g^{-1}(y)$ on $[1 - \eta, 1]$. It follows that $f(x) = g(x)$ on some neighborhood $[0, \epsilon]$ of $0$, and from complex analyticity of $f(x)$ and $g(x)$, then $f(x) = g(x)$ in some neighborhood of 0 in $\C$. This implies that $\tilde{f}(t) =  \sum_{k = - \infty}^{\infty} a_k t^k$ and $\tilde{g}(t) = \sum_{k  = - \infty}^{\infty} b_k t^k$ are equal on the unit circle. Finally, for every $k$,
	\begin{align*}
	a_k = \frac{1}{2 \pi} \int_{\Abs{z} = 1}^{} \frac{\tilde{f}(z)}{z^{k + 1}} \; dz 
	= \frac{1}{2 \pi} \int_{\Abs{z} = 1}^{} \frac{\tilde{g}(z)}{z^{k + 1}} \; dz 
	= b_k
	\end{align*}
\end{proof}

To show that $(f^{-1})'(y) = (g^{-1})'(y)$ on $(1 - \eta, 1)$, we first provide a useful estimate involving the beta function: 

\begin{lemma}
	\label{lem:beta_estimate}
	For $x$, $y > 0$,  $y$ fixed, and $x \to \infty$,
	\begin{align*}
	\int_{0}^{1} t^{x - 1} (1 - t)^{y - 1} \; dt \sim \Gamma(y) x^{-y} \; ,
	\end{align*}
	where $\sim$ denotes that the quotient of the quantities goes to 1 as $x \to \infty$. 
\end{lemma}
\begin{proof}
	The integral on the left hand side is the beta function, $B(x, y)$. A defining characteristic of the beta function, proved in \cite[p.~253-256]{whit2003}, is  
	\begin{align*}
	B(x, y) = \frac{\Gamma(x) \Gamma(y)}{\Gamma(x + y)} \;.
	\end{align*}
	An equivalent form of Stirling's approximation which appears in \cite[p.~253]{whit2003} is:
	\begin{align*}
	\Gamma(t + 1) \sim \sqrt{2 \pi t} \left( \frac{t}{e} \right)^t \;,
	\end{align*}
	Applying Stirling's approximation to the terms $\Gamma(x)$ and $\Gamma(x + y)$ in the formula for $B(x, y)$ as $x \to \infty$ produces the result. 
\end{proof}

Finally, we will prove that if the asymptotics in Proposition \ref{prop:estimate_2} hold, then $(f^{-1})'(y) = (g^{-1})'(y)$ on $(1 - \eta, 1)$:

\begin{proposition}
	\label{prop:estimate_3}
	Let $\{a_n\}$ and $\{b_n\}$ be symmetric, primitive random walks with the same return probabilities.  Let $f(x) = \sum_{k = - \infty}^{ \infty} a_k e^{i k x}$ and $g(x) = \sum_{k = - \infty}^{\infty} b_k e^{i k x}$ have inverses that are continuous in $[1 - \eta, 1]$ and continuously differentiable in $(1 - \eta, 1)$ for some $\eta \in (0, 1)$. If 
	\begin{align*}
		 \int_{1 - \eta}^{1} y^{2n}  \left( (f^{-1}) ' (y) - (g^{-1}) '(y) \right)  \; dy = o(e^{- 2 \mu n}) 
	\end{align*}
	for some $\mu > 0$, then $(f^{-1})'(y) = (g^{-1})'(y)$ on $(1 - \eta, 1)$.
\end{proposition}
\begin{proof}
	First, note that for any $\eta \in (0, 1)$, 
	\begin{align}
	\label{eqn:prop_estimate_3_1}
	\int_{\eta}^{1} (1 - u)^{2n} u^m \; du = o(e^{- \lambda n})
	\end{align}
	for some $\lambda > 0$. 
	
	For the sake of contradiction, assume that $(f^{-1})'(y) \neq (g^{-1})'(y)$. From Proposition \ref{prop:h_function}, $(f^{-1})'(y) = (1 - y)^{- 1 /2} q(y)$ and $(g^{-1})'(y) = (1 - y)^{- 1/ 2} r(y)$, for $q(y)$ and $r(y)$ analytic near $1$. Thus, $q(y) - r(y) \neq 0$, so that $q(y) - r(y) = \sum_{k = 0}^{\infty} b_k (y - 1)^k$, where $b_m \neq 0$ for smallest such $m$.
	
	Changing variables $y = 1- u$, we have the following equalities:
	\begin{align*}
		\int_{1 - \eta}^{1} y^{2n}  \left( (f^{-1}) ' (y) - (g^{-1}) '(y) \right)  \; dy 
		& = \int_{0}^{\eta} (1 - u)^{2n} u^{- 1 / 2} (q(1 - u) - r(1 - u)) \; du 
		\\ &= \int_{0}^{\eta} (1 - u)^{2n} u^{- 1 / 2} \sum_{k = 0}^{\infty} (-1)^k b_k u^k \; du
	\end{align*}
	 
	 Let $p(u) = \sum_{k = m +1}^{\infty} (-1)^k b_k u^{k - m - 1}$, which is analytic near $0$. Let $\Abs{p(u)} \leq M$ in $[0, \eta]$. From repeated application of the triangle inequality and (\ref{eqn:prop_estimate_3_1}), we obtain the following estimate: 
	 \begin{align*}
	 & \Abs{\int_{0}^{\eta} (1 - u)^{2n} u^{- 1/2} \sum_{k = 0}^{\infty} (-1)^k b_k u^k \; du}
	 \\ & \geq \Abs{b_m} \Abs{\int_{0}^{1} (1 - u)^{2n} u^{m - 1/2} \; du } - M \Abs{\int_{0}^{1} (1 - u) u^{m + 1/2}  \; du} - o(e^{- \lambda n})
	 \end{align*}
	 
	 It follows from Lemma \ref{lem:beta_estimate} that as $n \to \infty$, 
	 \begin{align*}
	 (2n + 1)^{m + 1/2}  \Abs{\int_{1 - \eta}^{1} y^{2n}  \left( (f^{-1}) ' (y) - (g^{-1}) '(y) \right)  \; dy} \geq \Abs{b_m}
	 \end{align*}
	 
	 Finally, as $n \to \infty$, $(2n + 1)^{- m - 1/2} e^{2 \mu n} \to \infty$ for any $\mu > 0$. Thus, as $n \to \infty$, we would have 
	 \begin{align*}
	 & e^{2 \mu n} \Abs{\int_{1 - \eta}^{1} y^{2n}  \left( (f^{-1}) ' (y) - (g^{-1}) '(y) \right)  \; dy} 
	 \\ & = \left( e^{2 \mu n}  (2n + 1)^{- m - 1/2} \right) (2n + 1)^{m + 1/2}  \Abs{\int_{1 - \eta}^{1} y^n  \left( (f^{-1}) ' (y) - (g^{-1}) '(y) \right)  \; dy} 
	 \\ & \geq \Abs{b_m} \left( e^{2 \mu n}  (2n + 1)^{- m - 1/2} \right)
	 \to \infty \; ,
	 \end{align*}
	 contradicting our hypothesis. Hence, $(g^{-1})'(y) = (f^{-1})'(y)$ on $(1 - \eta, 1)$. 
\end{proof}

\section{Proof of Corollary \ref{cor:rep_theory}}
Before we prove Corollary \ref{cor:rep_theory}, we will first prove a fact about the return probabilities of symmetric random walks.

\begin{proposition}
	\label{prop:rp_asymptotics}
	Let $\{a_n\}$ be a symmetric random walk. Let $c_n$ be the return probabilities of $\{a_n\}$. Then, $\limsup_{n \to \infty} \; c_{2n}^{\enspace 1/(2n)} = 1$. 
\end{proposition}
\begin{proof}
	Let $r_{n, k} = \mathrm{Pr}[x_1 + \cdots + x_n = k]$. Since $\{a_n\}$ is symmetric, then $r_{n, k} = r_{n, -k}$. Thus,
	\begin{align*}
	c_{2n} = r_{2n, 0} = \sum_{k = - \infty}^{\infty} r_{n, k} r_{n, -k}
	= \sum_{k = - \infty}^{\infty} r_{n, k}^2 
	\end{align*} 
	
	Let $y_n := (x_1 + \cdots + x_n) / N$, where the $x_i$ are i.i.d. discrete random variables with $\mathrm{Pr}[x_i = n] = a_n$. It follows that $E[y_n] = 0$ and $\sigma^2 := \Var(y_n) = \Var(x_i)  = \sum_{k = - \infty}^{\infty} k^2 a_k < + \infty$. It follows from Chebyshev's inequality that
	\begin{align*}
	\mathrm{Pr}[\Abs{x_1 + \cdots + x_n} > n^2] = \mathrm{Pr}[\Abs{y_n} > n] \leq \frac{\sigma^2}{n^2}
	\end{align*}
	Hence, 
	\begin{align*}
	\sum_{k = - n^2}^{n^2} r_{n, k} = \sum_{k = - \infty}^{\infty} r_{n, k}  \; - \; \sum_{\Abs{k} > n^2}^{} r_{n, k}
	= 1 - \mathrm{Pr}[\Abs{x_1 + \cdots + x_n} > n^2] \geq 1 - \frac{\sigma^2}{n^2}
	\end{align*}
	Applying the Cauchy-Schwartz inequality to the vectors $(r_{n, - n^2}, r_{n, - n^2 + 1}, \ldots, r_{n, n^2}), (1, 1, \ldots, 1) \in \R^{2 n^2 + 1}$, we have a lower bound for $c_{2n}$: 
	\begin{align*}
	c_{2n} = \sum_{k = - \infty}^{\infty} (r_{n, k})^2 
	\geq \sum_{k = - n^2}^{n^2} (r_{n, k})^2 
	\geq \left( \sum_{k = - n^2}^{n^2} r_{n, k} \right) / (2 n^2 + 1)
	\geq \frac{1 - \sigma^2 / n^2}{2 n^2 + 1}
	\end{align*}
	For $n$ sufficiently large, $ 1 - \sigma^2 / n^2 > 1/2$. For an upper bound on $c_{2n}$, $c_{2n} = \sum_{k = - \infty}^{\infty} (r_{n, k})^2 \leq \sum_{k = \infty}^{\infty} r_{n, k} = 1$.  Hence, for $n$ sufficiently large, 
	\begin{align*}
	1 \geq c_{2n}^{\enspace 1/(2n)} \geq \left( \frac{1}{4n^2 + 2}\right)^{1/(2n)}
	\end{align*}
	Taking the limsup as $n \to \infty$, then $\limsup_{n \to \infty} \; c_{2n}^{\enspace 1/(2n)} = 1$.
\end{proof}

%\begin{proposition}
%	\label{prop:rp_asymptotics}
%	Let $\{a_n\}$ be a symmetric random walk, where only finitely many of the $a_n$ are non-zero. Let $c_n$ be the return probabilities of $\{a_n\}$. Then, $\limsup_{n \to \infty} (c_{2n})^{1/(2n)} = 1$. 
%\end{proposition}
%\begin{proof}
%	Suppose that $a_n = 0$ for $n > M$. Let $r_{n, k} = Pr[x_1 + \cdots + x_n = k]$. Note that $r_{n, k} = 0$ for $\Abs{k} > Mn$. It follows that $1 = \sum_{k = - \infty}^{\infty} r_{n, k} = \sum_{k = -Mn}^{Mn} r_{n, k}$. 
%	
%	Also, since $\{a_n\}$ is symmetric, then $r_{n, k} = r_{n, -k}$. Thus,
%	\begin{align*}
%	c_{2n} = r_{2n, 0} = \sum_{k = - \infty}^{\infty} r_{n, k} r_{n, -k}
%	= \sum_{k = - \infty}^{\infty} (r_{n, k})^2 
%	= \sum_{k = - Mn}^{Mn} (r_{n, k})^2
%	\end{align*} 
%	Applying the Cauchy-Schwartz inequality to the vectors $(r_{n, -Mn}, r_{n, -Mn + 1}, \ldots, r_{n, Mn}), (1, 1, \ldots, 1) \in \R^{2 Mn + 1}$, we have a lower bound for $c_{2n}$:  
%	\begin{align*}
%	c_{2n} \geq \frac{1}{2 Mn + 1}
%	\end{align*}
%	Clearly, $c_{2n} = \sum_{k = - Mn}^{Mn} (r_{n, k})^2 \leq \sum_{k = -Mn}^{Mn} r_{n, k} = 1$. 
%	Hence, 
%	\begin{align*}
%	1 \geq (c_{2n})^{1/(2n)} \geq \left( \frac{1}{2 M n + 1}\right)^{1/(2n)}
%	\end{align*}
%	Taking the limsup as $n \to \infty$, then $\limsup_{n \to \infty} (c_{2n})^{1/(2n)} = 1$.
%\end{proof}

Next, we prove a lemma relating symmetry and primitivity of random walks to self-duality and faithfulness of $\mathrm{U}(1)$ representations, respectively. The term \textit{\textbf{isotypical decomposition}} in Lemma \ref{lem:reps_walks} refers to the unique decomposition of any finite-dimensional complex representation of a compact group into a direct sum of irreducible representations (see Prop 1.8 in \cite{fulton1991} for more details). 

\begin{lemma}
	\label{lem:reps_walks}
	Let $\rho: \mathrm{U}(1) \to \mathrm{GL}(V)$ be a $\mathrm{U}(1)$ representation, where $V$ is a finite-dimensional complex vector space with isotypical decomposition
	\begin{align*}
	V \isomorphic \bigoplus V_n^{\dsum \alpha_n} \; ,
	\end{align*}
	where $V_n$ is the 1-dimensional representation on which $z \in \mathrm{U}(1)$ acts by multiplication by $z^n$, and $\alpha_n \in \N$.
	Then, $\rho$ is self-dual if and only if $\alpha_n = \alpha_{-n}$ for every $n$, and $\rho$ is faithful if and only if $\{n : \alpha_n \neq 0 \}$ generates $\Z$.  
\end{lemma}
\begin{proof}
	As taking dual commutes with direct sum, then 
	\begin{align*}
	V^* \isomorphic \bigoplus (V_n^*)^{\dsum \alpha_{n}} = \bigoplus V_{-n}^{\dsum \alpha_n} = \bigoplus V_n^{\dsum \alpha_{-n}}
	\end{align*}
	As the isotypical decomposition is unique, then $V^* = V$ precisely when $\alpha_n = \alpha_{-n}$. 
	
	If $\rho$ is faithful, then $\{n : \alpha_n \neq 0\}$ must generate $\Z$, for if $p > 1$ divides all elements of the set, then any $p^{th}$ root of unity acts the same on $V$, contradicting faithfulness. On the other hand, if the set generates $\Z$, then $z$ can be recovered from the $z^{\alpha_n}$-action on the isotypic components of $V$. 
\end{proof}

Now, we can prove Corollary \ref{cor:rep_theory}: 
\begin{proof}[Proof of Corollary \ref{cor:rep_theory}]
	Let $\rho: \mathrm{U}(1) \to \mathrm{GL}(V)$ be a self-dual, faithful $\mathrm{U}(1)$ representation, where $V$ is a finite-dimensional complex vector space with isotypical decomposition
	\begin{align*}
	V \isomorphic \bigoplus V_n^{\dsum \alpha_n} \; ,
	\end{align*}
	where $V_n$ is the 1-dimensional representation on which $z \in \mathrm{U}(1)$ acts by multiplication by $z^n$, and $a_n \in \N$.
	
	It is evident that in this decomposition, $\dim(V)^{\mathrm{U}(1)} = a_0$. In general, using distributivity of direct sum and tensor product, and $V_n \tens V_m \isomorphic V_{n + m}$, we see that $\dim (V^{\tens n})^{\mathrm{U}(1)}$ is the constant coefficient of $(\sum \alpha_k x^k)^n$. If $N = \sum \alpha_k = \dim(V)$, then the constant coefficient of $(\sum \alpha_k x^k)^n$ is just $N^n c_n$, where $c_n$ is the $n^{th}$ return probability of the random walk $\{a_n\}$, where $a_n = \alpha_n / N$. From Lemma \ref{lem:reps_walks}, $\{a_n\}$ is a symmetric, primitive random walk. 
	
	Let $\sigma: \mathrm{U}(1) \to \mathrm{GL}(W)$ be another complex, finite-dimensional, self-dual, faithful $\mathrm{U}(1)$ representation, with isotypical decomposition
	\begin{align*}
	W \isomorphic \bigoplus V_n^{\dsum \beta_n} \; .
	\end{align*}
	If $M = \dim(W)$, then using the same argument as before, $\dim(W^{\tens n})^{\mathrm{U}(1)} = M^n d_n$, where $d_n$ is the $n^{th}$ return probability of a symmetric, primitive random walk $\{b_n\}$, with $b_n = \beta_n / M$.

	Suppose that $\dim(V^{\tens n})^{\mathrm{U}(1)} = \dim(W^{\tens n})^{\mathrm{U}(1)}$, for every $n$.
	To show that $\{a_n\} = \{b_n\}$, it suffices to show that $M = N$, for then $c_n = d_n$, and from Theorem \ref{thm:main}, $a_n = b_n$, so then $\alpha_n = \beta_n$ and $V \isomorphic W$. 
	
	The following equalities are true for every $n$: 
	\begin{align*}
	M d_n^{\enspace 1/n} = \left( \dim (W^{\tens n})^{\mathrm{U}(1)}\right)^{1/n}  = \left( \dim (V^{\tens n})^{\mathrm{U}(1)}\right)^{1/n} = N c_n^{\enspace 1/n} 
	\end{align*}
	Taking the limit as $n \to \infty$ for $n$ even and using Proposition \ref{prop:rp_asymptotics}, then $M = N$. 
\end{proof}

\section{Acknowledgements}
This research was completed during the author's undergraduate senior thesis at Indiana University Bloomington, supervised by Professor Michael Larsen. I would like to thank Professor Larsen for his guidance throughout the project and his help editing this paper.  

\bibliographystyle{plain}
\bibliography{return_prob}

\end{document}